\documentclass[12pt]{amsart}
\newcounter{defcounter}
\usepackage{amssymb,amsmath,amscd,graphicx, enumitem,
latexsym,amsthm}
\usepackage{amssymb,latexsym,eufrak,amsmath,amscd,graphicx, amsxtra}
  \usepackage[all]{xy}
  \usepackage{pdfsync}
  \usepackage[mathscr]{euscript}

\theoremstyle{plain}
\newtheorem{theorem}{Theorem}

\newtheorem{proposition}[theorem]{Proposition}

\newtheorem{lemma}[theorem]{Lemma}

\newtheorem*{theoremnn}{Theorem}
\newtheorem{proposition.definition}[theorem]{Proposition/Definition}

\newtheorem{theoremalpha}{Theorem}

\newtheorem*{lemmann}{Lemma}

\theoremstyle{definition}

\newtheorem*{examplenn}{Example}
\newtheorem{remark}[theorem]{Remark}
\newtheorem{example}[theorem]{Example}

\newcommand{\lra}{\longrightarrow}

\newcommand{\noi}{\noindent}
\newcommand{\PP}{\mathbf{P}}

\newcommand{\CC}{\mathbf{C}}

\newcommand{\Zero}[1]{\textnormal{Zeroes} ({#1})}
\newcommand{\OO}{\mathcal{O}}

\newcommand{\II}{\mathcal{I}}

\newcommand{\frb}{\mathfrak{b}}

\newcommand{\Image}{\textnormal{Im}}
\newcommand{\GGG}[2]{\Gamma \big( {#1},{#2}\big)}
\newcommand{\HH}[3]{H^{{#1}} \big( {#2} , {#3}
\big) }
\newcommand{\bHH}[3]{\mathbf{H}^{{#1}} \big( {#2} , {#3}
\big) }

\newcommand{\MI}[1]{\mathcal{J} \big( {#1}
\big) }

\newcommand{\coker}{\textnormal{coker}}
\newcommand{\MMI}[2]{\MI{ {#1} \, , \, {#2}}}

\newcommand{\pr}{\prime}

\newcommand{\lin}{\equiv_{\text{lin}}}

\newcommand{\dra}{\dashrightarrow}
\newcommand{\Bl}{\text{Bl}}

\newcommand{\Linser}[1]{| \mspace{1.5mu} {#1}
\mspace{1.5mu} |}
\newcommand{\linser}[1]{\Linser{  {#1}  }}

\newcommand{\pro}{\textnormal{pr}}

\newcommand{\res}{\textnormal{res}}

\numberwithin{theorem}{section}

\begin{document}

\title{Cayley-Bacharach Theorems with Excess Vanishing}
 \author{Lawrence Ein}
 \address{Department of Mathematics, University of Illinois at Chicago, 851 South Morgan St., Chicago, IL  60607}
  \email{{\tt ein@uic.edu}}
  \thanks{Research of the first author partially supported by NSF grant DMS-1801870.}
 
 \author{Robert Lazarsfeld}
  \address{Department of Mathematics, Stony Brook University, Stony Brook, New York 11794}
 \email{{\tt robert.lazarsfeld@stonybrook.edu}}
 \thanks{Research  of the second author partially supported by NSF grant DMS-1739285.}

\dedicatory{Dedicated to Bill Fulton on the occasion of his 80th birthday.}

\maketitle

 \section*{Introduction}
 
A classical theorem of Cayley and Bacharach asserts that if $D_1, D_2 \subseteq \PP^2$ are curves of degrees $d_1$ and $d_2$ meeting transversely,  then any curve of degree $d_1 + d_2 - 3$ passing through all but one of the $d_1d_2$ points of $D_1 \cap D_2$ must also contain the remaining point.  Generalizations of this statement have been a source of fascination for decades. Algebraically, the essential point is that complete intersection quotients of a polynomial ring are Gorenstein. We refer the reader to \cite[Part I]{EGH} for a detailed overview.

The most natural geometric setting for results of this sort was introduced in the paper \cite{GH} of Griffiths and Harris. Specifically, let $X$ be a smooth complex projective variety of dimension $n$,  let $E$ be a vector bundle on $X$ of rank $n$, and set $L = \det E$. Suppose given a section $s \in \GGG{X}{E}$ that vanishes simply along a finite set $Z \subseteq X$. Griffiths and Harris prove that if
\[ h \, \in \, \GGG{X}{\OO_X(K_X + L)} \]
vanishes at all but one of the points of $Z$, then it vanishes at the remaining point as well. This of course implies   statements for hypersurfaces in projective space by taking $E$ to be a direct sum of line bundles.

The starting point of the present note was the paper \cite{Li} of Mu-Lin Li, who proposed an extension allowing for excessive vanishing. With $X$, $E$ and $L$ as above, suppose that $s \in \GGG{X}{E}$ is a section that vanishes scheme-theoretically along a smooth subvariety $W \subseteq X$ of dimension $w \ge 0$ in addition to a non-empty reduced finite set $Z \subseteq X$:
\[
\Zero{s} \ = \ W \sqcup Z. 
\]
For example, one might imagine three surfaces in $\PP^3$ cutting out the union of a smooth curve and a finite set. Assuming for simplicity that $W$ is irreducible, its normal bundle $N_{W/X}$ sits naturally as a sub-bundle of the restriction $E\, | \, W$, giving rise to an exact sequence
\[  0 \lra N_{W/X} \lra E\, |  \, W \lra V \lra 0, \tag{*} \]
where $V$ is a vector bundle of rank $w $ on $W$. Li's result is the following:
\begin{theoremnn}[\cite{Li}, Corollary 1.3]
Assume that the exact sequence $(*)$ splits. Then any section of $\OO_X(K_X + L)$ vanishing on $W$ and at all but one of the points of $Z$ vanishes also on the remaining point of $Z$.
\end{theoremnn}
\noi His argument is analytic in nature, using what he calls ``virtual residues."

It is natural to ask whether the statement remains true without assuming the splitting of (*). The following example shows that this is  not the case.
\begin{examplenn}
Let $C \subseteq \PP^3$ be a rational normal cubic curve, and fix general surfaces
\[      Q_1\, , \, Q_2 \, , \, F \ \supseteq \ C \]
containing $C$,  with  $\deg Q_1 = \deg Q_2 = 2$ and $\deg F = d > 2$. Recall that $Q_1 \cap Q_2 = C \cup L$ where $L$ is a line meeting $C$ at two points $a_1, a_2 \in L$. \begin{figure}
\includegraphics[scale = 1]{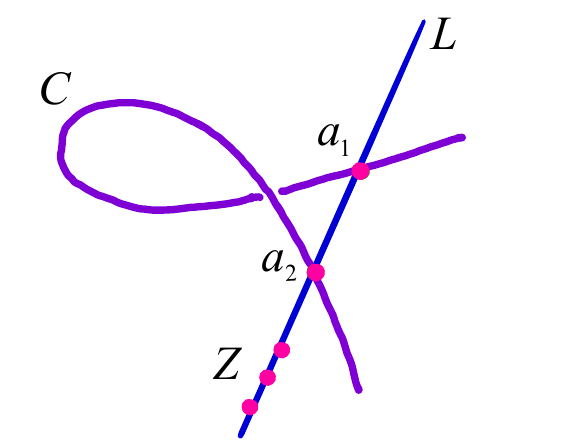}
\vskip -.1in
\caption{Twisted cubic and secant line: $C \cup L  = Q_1 \cap Q_2$}\label{CBExample}
\end{figure}
Therefore 
\[   Q_1 \cap Q_2 \cap F \ = \ C \, \sqcup Z, \]
where $Z$ consists of the $(d-2)$ additional points of intersection of $F$ with $L$. (See Figure \ref{CBExample}.) 
The conclusion of Li's theorem would be that any surface $H$ of degree $2 + 2 + d - 4 = d$ passing through $C$ and all but one of the points of $Z$ passes through the remaining one. However this need not happen: for instance, one can take $H$ to be the union of a general cubic through $C$ and $(d-3)$ planes each passing through exactly one of the points of $Z$. \qed
\end{examplenn}

On the other hand, staying in the setting of the Example, suppose that $H \subseteq \PP^3$ is a surface of degree $d$ that passes \textit{doubly} through $C$ and in addition contains all but one of the points of $Z$. Then $H$ meets $L$ twice at $a_1$ and $ a_2$ as well as  at $(d-3)$ other points, and therefore $H \supseteq L$. In other words, in this case the conclusion of Li's theorem does hold if one looks at surfaces that have multiplicity $\ge 2$ along $C$. This is an illustration of our first general result.

\begin{theoremalpha}\label{ExcessCB.Intro.A}
With $X$ and $E$ and $L = \det E$ as above, consider a section  $s \in \GGG{X}{E}$ with
\[
\Zero{s} \ =_{\textnormal{scheme-theoretically}} \ W \, \sqcup \, Z, 
\]
where $W$ is smooth of dimension $w$ and $Z$ is a non-empty reduced finite set. Suppose  that 
\[   h \, \in \, \GGG{X}{\OO_X(K_X + L) \otimes I_{W}^{w+1}}\]
is a section of $K_X + L$ vanishing to order $(w+1)$ along $W$, as well as at all but one of the points of $Z$. Then $h$ vanishes also at the remaining point of $Z$. 
 \end{theoremalpha}
 \noi Note that if $\dim W = 0$, this reduces to the classical result.  See also Example \ref{Li-Type.Statement} for an application to statements closer to the spirit of \cite{Li}. 
 
 Theorem \ref{ExcessCB.Intro.A} is a special case of a more general result involving multiplier ideals. Continuing to keep $X, E$ and $L$ as before, consider a section $s \in \GGG{X}{E}$ that vanishes simply along a non-empty finite set $Z \subseteq X$ and arbitrarily along a scheme disjoint from $Z$ defined by an ideal $\frb \subseteq \OO_X$. In other words, we ask that that the image of the map
 \[    E^*   \lra \OO_X \]
 defined by $s$ be the ideal $\frb \cdot I_Z$, with $\frb + I_Z = \OO_X$. One can associate to $\frb$ and its powers \textit{multiplier ideals} $\MI{\frb^m} = \MMI{X}{\frb^m} \subseteq \OO_X$ that measure in a somewhat delicate way the singularities of elements of $\frb$. We prove:
 \begin{theoremalpha} \label{ExcessCB.Intro.B}
 Let 
 \[  h \, \in \, \Gamma \Big( \OO_X(K_X + L) \otimes \MI{\frb^n} \Big) \]
 be a section of $\OO_X(K_X + L)$ vanishing along the multiplier ideal $\MI{\frb^n}$, and suppose that $h$ vanishes at all but one of the points of $Z$. Then it vanishes also at the remaining point. 
  \end{theoremalpha}
  \noi If $\frb = I_W$ is the ideal sheaf of a smooth subvariety of dimension $w$, then $\MI{\frb^n} = I_W^{w+1}$, yielding Theorem \ref{ExcessCB.Intro.A}. We remark that it is not essential that $s$ vanish simply along the finite set $Z$, but then one has to reformulate (in a well-understood   manner) what it means for $h$ to vanish at all but one of the points of $Z$: see Remarks \ref{Non.Reduced.Zeroes} and \ref{Non.Red.Excess}.  
  
   Theorem \ref{ExcessCB.Intro.B} follows almost immediately from the classical statement, but at the risk of making the result seem more subtle than it is let us explain conceptually why one expects multiplier ideals to enter the picture. When $\Zero{s} = Z$ is a finite set, one can think of Cayley-Bacharach as arising via duality from the exactness of the Koszul  complex
   \[  0 \lra \Lambda^n E^*  \lra \Lambda^{n-1} E^* 
   \lra \ldots \lra \Lambda^2 E^*   \lra E^*   \lra \II_Z \lra 0 \tag{Kos} \]  
  determined by $s$. (See \S1 for a review of the argument, which is due to Griffiths--Harris.) If $s$ vanishes excessively this complex is no longer exact, which is why -- as in the example above -- the most naive analogue of Cayley-Bacharach fails. However   (Kos) contains a subcomplex involving multiplier ideals that  always is exact:
 \[ 0 \lra \Lambda^n E^* \lra \Lambda^{n-1} E^* \otimes \MI{\frb} \lra \ldots \lra  E^* \otimes \MI{\frb^{n-1}} \lra \MI{\frb^n} \cdot \II_Z \lra 0. \tag{Skod} \]
 (This is essentially the \textit{Skoda complex} introduced in \cite{EL}: see Example \ref{Skoda.Complex} below.) 
 One can view Theorem  \ref{ExcessCB.Intro.B} as coming from (Skod) in much the same way that the classical result arises from (Koz). 
 
There are  variants of the Griffiths--Harris theorem that also extend to the setting of excess vanishing. As above, let $s \in \GGG{X}{E} $ be a section that vanishes simply on a finite set $Z$. Tan and Viehweg \cite{TV} in effect prove the following:
\begin{theoremnn} 
Fix an arbitrary line bundle $A$ on $X$, and 
write $Z = Z_1 \sqcup Z_2$ as the  union of two disjoint non-empty subsets. Set\[ \begin{gathered}
v_1 \ = \ \dim  \coker \Big( H^0(A) \lra H^0\big(A \otimes \OO_{Z_1}\big) \Big) \\
v_2 \ = \ \dim \coker \Big( H^0\big( I_Z (K_X + L - A)\big) \hookrightarrow H^0\big(I_{Z_2}(K_X + L - A) \big) \Big),
\end{gathered}
\]
so that $v_1$ measures the failure of $Z_1$ to impose independent conditions on $H^0(A)$, while $v_2$ counts the number of sections of $\OO_X(K_X + L - A)$ vanishing on $Z_2$ but not on $Z_1$. 
Then
\begin{equation}\label{Viehweg.Tan.Eqn.1}
v_2 \, \le \, v_1.  \end{equation}
\end{theoremnn}
\noi So for example, if $A = \OO_X$ and $Z_1 $ consists of  single point $x \in Z$, then $v_1 = 0$ and this reduces to the classical statement. Similarly, if we choose $Z_1$ in such a way that it imposes independent conditions on $H^0(A)$ then the assertion is that any section of $\OO_X(K_X + L - A)$ vanishing on $Z_2 = Z - Z_1$ also vanishes on $Z_1$. The theorem of Tan--Viehweg generalizes  analogous statements for hypersurfaces in projective space (\cite{Bach}, \cite{DGO}, \cite{EGH}).\footnote{In the classical case, vanishings for the cohomology of line bundles on projective space yield the stronger assertion that $v_2 = v_1$.}

We prove that in the case of possibly excessive vanishing, the analogous statement remains true taking  into account multiplier ideal corrections. 
\begin{theoremalpha} \label{CB.Excess.TV.Intro}
Suppose as above that $s$ defines the ideal $\frb \cdot I_Z \subseteq \OO_X$. Then the inequality \eqref{Viehweg.Tan.Eqn.1} continues to hold  provided that one takes
\[  
v_2 \ = \ \dim \coker \Big( H^0\big( I_Z (K_X + L - A)\otimes\MI{\frb^n} \big) \hookrightarrow H^0\big(I_{Z_2}(K_X + L - A) \otimes\MI{\frb^n}\big) \Big)
\]
\end{theoremalpha}
\noi Again this follows quite directly from the classical statement. 

Concerning the organization of this note, we start in \S1 with a review of the theorem of Griffiths--Harris, and the  extension in the spirit of Tan and Viehweg. 
As an application of the latter, we give at the end of the section a somewhat simplified and strengthened account of some results of Sun \cite{Sun} concerning finite determinantal loci: see Theorem \ref{CB.for.Det.Loci}.\footnote{In an earlier version of this paper, we overlooked the work of Sun. We apologize for this ommission.}    In \S2 we derive the results involving multiplier ideals by applying the classical theorems on a log resolution of the base ideal.  Since our primary interests lie on excess vanishing, we make the simplifying assumption throughout the main exposition that the finite zero-locus $Z$ is reduced. The well-understood modifications needed in the general case are discussed in Remarks \ref{Non.Reduced.Zeroes}, \ref{Caveat} and \ref{Non.Red.Excess}.    We work throughout over the complex numbers.

\section{A Review of Cayley-Bacharach with Proper Vanishing}

\numberwithin{equation}{section}
\setcounter{equation}{0}

In this section we review the classical theorem of Cayley-Bacharach from the viewpoint Griffiths--Harris and its extension in the spirit Tan and Viehweg. As an application, we give at the end of the section some results of Cayley-Bacharach type for degeneracy loci.

Suppose then that $X$ is a smooth complex projective variety of dimension $n$, and let $E$ be a vector bundle of rank $n$ on $X$, with $\det E = L$. We assume given a section $s \in \GGG{X}{E}$ vanishing simply on a non-empty finite set $Z \subseteq X$. Thus
\[   \# Z \ = \ \int_X c_n(E). \]
In this setting, the basic result is due to Griffiths and Harris \cite{GH}:
\begin{theorem} \label{GH.CB.Thm}
Consider a section  
\[ h\,  \in \, \GGG{X}{\OO_X(K_X + L)}\] vanishing  at all but one of the points of $Z$. Then $h$ vanishes on the remaining one as well. 
\end{theorem}

We outline the argument of Griffiths and Harris from \cite{GH}. The starting point of the proof is to form the Koszul complex determined by $s$:
\[
0 \lra \Lambda^n E^* \lra \Lambda^{n-1}E^* \lra \ldots \ \lra E^* \lra \OO_X \lra \OO_Z \lra 0.
\]
Because $X$ is smooth and $s$ vanishes in the expected codimension $n$, this is exact. Now twist through by $\OO_X(K_X + L)$. Recalling that $L = \Lambda^nE$ this gives a long exact sequence:
\begin{multline} \label{Twisted.Koszul.Cx.1}
0\lra  \OO_X(K_X) \lra \Lambda^{n-1}E^*\otimes \OO_X(K_+L) \lra\ldots\\
\ldots \lra E^* \otimes \OO_X(K_X + L ) \lra \OO_X(K_X + L) \lra \OO_Z(K_K + L) \lra 0. \end{multline}
Splitting this into short exact sequences and taking  cohomology, one arrives at maps
\begin{equation} \label{CB.Proof.Equation}
\HH{0}{X}{\OO_X(K_X + L) } \lra \HH{0}{Z}{\OO_Z(K_X + L)} \overset{\delta}{\lra} \HH{n}{X}{\OO_X(K_X)}  \end{equation}
whose composition is zero. (Absent additional vanishings,  \eqref{CB.Proof.Equation} might not be exact.)

On the other hand, note that $L \otimes \OO_Z = \det(N_{Z/X})$ and hence there is a natural identification 
\[  \OO_Z(K_X + L) \ = \ \omega_Z. \]
Therefore duality canonically identifies $\delta$ with a homomorphism
\[   \delta^\pr : \HH{0}{Z}{\OO_Z}^*\lra \HH{0}{X}{\OO_X}^*. \tag{*} \]
Not surprisingly, one has the:
\begin{lemmann} The mapping $\delta^\pr$ in $(*)$ is the dual of the canonical restriction
\[  \HH{0}{X}{\OO_X} \lra \HH{0}{Z}{\OO_Z}.\]
\end{lemmann}
\noi As Griffiths and Harris observe, this is ultimately a consequence of the functoriality of duality. We give the proof of a more general result -- Lemma \ref{A.Dual.Lemma} below -- in Appendix \ref{Appendix.A}.

 Granting the Lemma, Theorem \ref{GH.CB.Thm} follows at once. In fact, in terms of the natural basis for $\HH{0}{Z}{\OO_Z}$ and its dual, the Lemma shows that $\delta^\pr$ is given by the matrix $(1,\ldots, 1)$. Therefore if $h \in \HH{0}{X}{\OO_X(K_X + L)}$ were to vanish at all but one of the points of $Z$ but not at the remaining one, then $h | Z \not \in \ker (\delta)$, contradicting the fact that the composition in \eqref{CB.Proof.Equation} is the zero mapping.

We turn now to a result in the spirit of  Tan and Viehweg \cite{TV}.\footnote{The actual statement and proof in  \cite{TV} are rather more complicated, but Theorem \ref{TV.Type.Statement} is essentially what is established there.  In \cite{Tan}, Tan relates these statements to the Fujita conjecture.}
 Keeping  assumptions and notation as above, write $Z = Z_1 \sqcup Z_2$ as the union of two non-empty subsets, and fix an arbitrary line bundle $A$.  Recall the statement:
\begin{theorem} \label{TV.Type.Statement}
Define
\[ \begin{gathered}
V_1 \ = \ \coker \Big( H^0(A) \lra H^0\big(A \otimes \OO_{Z_1}\big) \Big),  \\
V_2 \ = \ \coker \Big( H^0\big( I_Z (K_X + L - A)\big) \hookrightarrow H^0\big(I_{Z_2}(K_X + L - A) \big) \Big)
\end{gathered}
\]
Then
$\dim V_2 \, \le \, \dim  V_1.$  \end{theorem}
\noi As noted in the Introduction, this implies Theorem \ref{GH.CB.Thm} (at least when $\# Z \ge 2$). 

For the proof, one starts by tensoring \eqref{Twisted.Koszul.Cx.1} by $\OO_X(-A)$. Taking cohomology as before, one arrives at a complex
\begin{equation} \label{Complex.with.A.1}
H^0\big(\OO_X(K_X + L-A) \big) \overset{\res} \lra H^0 \big( \OO_Z(K_X + L-A)\big) \overset{\delta}{\lra} H^n 
\big(\OO_X(K_X-A)\big).
\end{equation}
Moreover, via the decomposition
\[H^0 \big(\OO_Z(K_X + L-A)\big) \ = \ H^0 \big(\OO_{Z_1}(K_X + L-A)\big) \, \oplus \, H^0 \big(\OO_{Z_2}(K_X + L-A)\big),
\]
this restricts to a subcomplex
\begin{equation} \label{Complex.with.A.2}
H^0 \big( \OO_X(K_X + L-A)\otimes I_{Z_2} \big) \overset{\res_1} \lra H^0 \big( \OO_{Z_1}(K_X + L-A)\big)\overset{\delta_1}{\lra} H^n \big(\OO_X(K_X-A)\big).
\end{equation}
Note next that 
\[  \ker (\res_1) \ = \ H^0 \big((K_X + L-A)\otimes I_{Z}\big),
\] 
and hence
$
V_2= \Image ( \res_1) . 
$
On the other hand, since 
\eqref{Complex.with.A.2} is a complex, one has
\[  
\dim \Image(\res_1) \ \le \ \dim \ker(\delta_1).  \] 
It is therefore sufficient to show that 
\[
\dim \ker(\delta_1) \ = \ \dim
 \coker \Big( H^0(A) \lra H^0\big(A \otimes \OO_{Z_1}\big) \Big).\tag{*}\]

For this we again apply duality, which 
  identifies $\delta$ and $\delta_1$ with a diagram of maps
  \[
  \xymatrix@R=7pt{
  \HH{0}{Z}{A \otimes \OO_Z}^* \ar[dr]^{\delta^\pr}   \\  &\HH{0}{X}{A}^* .\\
  \HH{0}{Z_1}{A \otimes \OO_{Z_1}}^* \ar@{^{(}->}[uu]\ar[ur]_{\delta_1^\pr}    
  }
  \]
As above the crucial point is to verify:
\begin{lemma} \label{A.Dual.Lemma}
The mappings $\delta^\pr$ and $\delta_1^\pr$ are dual to the natural restriction morphisms
\[   \HH{0}{X}{A} \lra \HH{0}{Z}{A \otimes \OO_Z } \ \ , \ \  \HH{0}{X}{A} \lra \HH{0}{Z_1}{A \otimes \OO_{Z_1} }. \]\end{lemma}
\noi A proof of the Lemma appears in Appendix \ref{Appendix.A}.  The Lemma implies  that in fact \[   \ker \delta_1 \ = \ V_1^*, \]
and Theorem \ref{TV.Type.Statement} is proved. 

\begin{remark} \textbf{(Non-reduced zero schemes).} \label{Non.Reduced.Zeroes}
It is not necessary to assume that the finite scheme $Z$ be reduced. In fact, since $Z$  is Gorenstein, one can associate to any subscheme $Z_1 \subseteq Z$ a residual scheme $Z_2 \subseteq Z$ having various natural properties: see \cite[p.\ 311 ff]{EGH} for a nice discussion. The hypothesis in Theorem \ref{GH.CB.Thm} should then be that $h$ vanishes on the scheme residual to a point $x \in Z$. In Theorem \ref{TV.Type.Statement} one works with a residual pair $Z_1, Z_2 \subseteq Z$. In this more general setting, one no longer has the embedding $\OO_{Z_1} \subseteq \OO_Z$ used in the proof. Instead, one replaces this with the canonical inclusion
\[ \omega_{Z_1} \otimes \omega_Z^*  \, \hookrightarrow \,\OO_Z, \]
and then the argument with duality goes through. We leave details to the interested reader. \qed  \end{remark}

We conclude this section by sketching an application of Theorem \ref{TV.Type.Statement} to statements, essentially due to Sun \cite{Sun},  of Cayley-Bacharach type for determinantal loci. The present approach is somewhat different than that of \cite{Sun}, which uses Eagon-Northcott complexes.

We start with the set-up. Let $X$ be a smooth projective variety of dimension $n$, and let $E$ be a vector bundle on $X$ of rank $n + e$ for some $e \ge 0$. Suppose given sections 
\[   s_0, \ldots, s_e \, \in \, \Gamma \big( X, E \big) \]
that drop rank simply along a reduced finite set $Z$, so that once again $\#Z  = \int c_n(E)$. Denote by $ W \subseteq H^0(E)$ the $(e+1)$-dimensional subspace spanned by the $s_i$, and write $V = W^*$ for the dual of $W$. The $s_i$ determine a natural vector bundle map $w: W_X \lra E$, where $W_X = W \otimes_{\CC} \OO_X$ is the trivial vector bundle with fibre $W$. By assumption $w$ has rank exactly $e$ at each point of $Z$, and hence its dual determines an exact sequence
\begin{equation} \label{Det.Locus.Eqn.1}  E^* \overset{ u} \lra V_X \lra B _Z \lra 0.  \end{equation}
where $V_X = V \otimes_{\CC} \OO_X$ and $B_Z$ is a line bundle on $Z$. In particular, there is a natural mapping 
\[  \phi \, = \, \phi_u \,  : \, Z \lra \PP(V) = \PP^e. \] 
More concretely, $\phi$ sends each point $z \in Z$ to the one-dimensional quotient $\coker \big( u(z)\big)$ of $V$. In particular, for any subset $Z^\pr \subseteq Z$, and any $k \ge 0$, one gets a homomorphism
\[ \rho_{Z^\pr, k} \, : \, H^0\big(\OO_{\PP}(k)\big)\lra H^0\big( \phi_* \OO_{Z^\pr}(k) \big ) \, = \, H^0 \big (Z^\pr, B^{\otimes k}_Z \, |\, {Z^\pr} \big).\]
Equivalently, this is the mapping
\[   S^k V \lra \HH{0}{Z^\pr}{B^{\otimes k}_Z \mid Z^\pr} \]
arising from \eqref{Det.Locus.Eqn.1}.

\begin{theorem} \label{CB.for.Det.Loci}
In the situation just described, set $L  = \det E$, and write $Z = Z_1 \sqcup Z_2$ as the disjoint union of two non-empty subsets. Define
\begin{equation} \label{Def.c1.c2.CB.Det.Loci}
\begin{gathered}
c_1 \ = \ \dim \coker \Big(   \rho_{Z_1, n-1}: H^0\big(\OO_{\PP}(n-1)\big)\lra H^0\big( \phi_* \OO_{Z_1}(n-1) \big )\Big)\\ 
c_2 \ = \  \dim \coker \Big( H^0\big(X,  I_Z (K_X + L)\big) \hookrightarrow H^0\big(X, I_{Z_2}(K_X + L) \big) \Big)\end{gathered}
\end{equation}
Then $c_2 \le c_1$.
\end{theorem}

\begin{example}
Let $C , D \subseteq \PP^2$ denote respectively a cubic and a quartic curve meeting transversely at twelve points. Take $O \in C \cap D$, and set \[Z \ =\  (C \cap D) - \{ O \}, \] so that $Z$ consists of eleven of the twelve intersection points of $C$ and $D$. Then $Z$ is the degeneracy locus of a map
\[   
 \OO^2_{\PP^2} \overset{w} \lra \OO_{\PP^2}(1) \oplus \OO_{\PP^2}(2) \oplus \OO_{\PP^2}(3), 
\footnote{If $O$ is defined  by linear forms $L_1$ and $L_2$, then the equations defining $C$ and $D$  are expressed as
\[ A_1L_1 - A_2 L_2 \ \ , \ \ B_1L_1 -B_2L_2, \]
where $\deg A_i = 2$, $\deg B_i = 3$. The map $w$ is then given by the matrix
$
\left (
\begin{matrix}
L_1&  L_2 \\ A_2 & A_1 \\ B_2 & B_1
\end{matrix}
\right ).
$
  }
 \]
whose dual  \eqref{Det.Locus.Eqn.1} has the form
\[ \OO_{\PP^2}(-1) \oplus \OO_{\PP^2}(-2) \oplus \OO_{\PP^2}(-3) \lra \OO^2_{\PP^2}. \]
Now pick two points $P, Q \in Z$ and take $Z_1 = \{ P, Q \}$. If the line joining $P$ and $Q$ passes through $O$, then $c_1 = 1$, otherwise $c_1 = 0$. According to the the Theorem, in the former case there may be an additional cubic passing through the remaining nine points of $Z$, but in the latter case there is none. If we take $D$ to be the union of a cubic $C^\pr$ and a line $L$, and if $P, Q, O \in L$, then in fact $c_2 = 1$. (See Figure \ref{DetExample}.) \qed
\begin{figure}
\includegraphics[scale = .18]{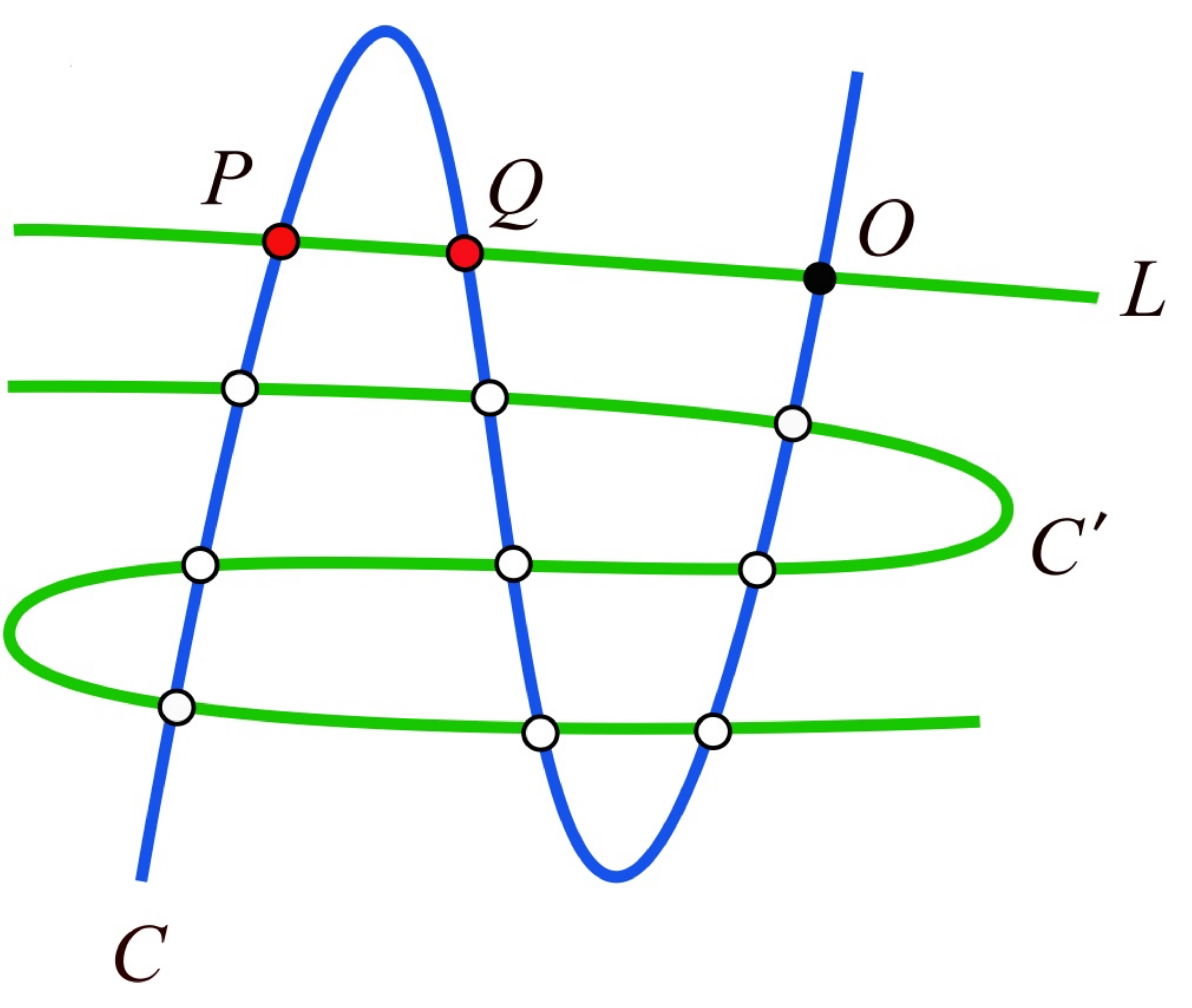}
\caption{Eleven points in intersection of cubic and quartic}\label{DetExample}
\end{figure}
\end{example}

\begin{proof}[Proof of Theorem \ref{CB.for.Det.Loci}] Set
$  Y= \PP(V_X) = X \times \PP(V)$, 
with projections 
\[ \pro_1 : Y \lra X \ \ , \ \ \pro_2 : Y \lra \PP(V). \]
The plan is to use a well-known construction to realize $Z$ as the zero-locus of a section of a vector bundle on $Y$, and then apply Theorem \ref{TV.Type.Statement}. Specifically, the presentation \eqref{Det.Locus.Eqn.1} determines an embedding $Z \subseteq Y$ under which $B_Z = \pro_2^* \big(\OO_{\PP(V)}(1)\big)|Z$. Moreover, $Z$ is defined in $Y$ by the vanishing of the composition 
\[
\xymatrix{\pro_1^* \, E^* \ar[r]^{\pro_1^*(u)} \ar[dr]&\pro_1^* V_X \ar[d] \\
&\pro_2^* \, \OO_{\PP(V)}(1).
} \]
In other words, writing 
\[ F \ = \  \pro_1^*E \otimes \pro_2^* \OO_{\PP(V)}(1), \] $Z \subseteq Y$ is the zero-locus of a section
 $   s \in \Gamma(Y,F)$.
Note next that
\[   K_Y + \det F  \ = \ \pro_1^*\big( \OO_X( K_X + \det E \big) \big) \, \otimes\, \pro_2^* \OO_{\PP(V)}(n-1). 
\]
We now apply Theorem \ref{TV.Type.Statement} with  $A = \pro_2^*\OO_{\PP(V)}(n-1)$. Then on the one hand, for any $Z^\pr \subseteq Z$, the restriction
\[  \HH{0}{Y}{A}\lra  \HH{0}{Z^\pr}{A|Z^\pr} \]
is identified with the map $\rho_{Z^\pr , n-1}$ appearing in \eqref{Def.c1.c2.CB.Det.Loci}. On the other hand, since 
\[  K_Y + \det F  - A 
\ = \ \pro_1^* \big( K_X + \det E\big),\]
for any $Z^\pr \subseteq Z$ one has
\[ 
\HH{0}{Y}{\OO_Y(K_Y + \det F - A) \otimes I_{Z^\pr / Y}} \ = \ \HH{0}{X}{\OO_X(K_X + \det E) \otimes I_{Z^\pr / X}}.\]
The Theorem follows. 
\end{proof}

\begin{example}
When $Z_1$ consists of a single point in $Z$, we find that if $s_0, \ldots, s_{e} \in \Gamma \big(X, E\big)$  drop rank along $Z$, then any section of $\OO_X(K_X + L)$ vanishing at all but one of the points of $Z$ vanishes at the remaining one. (This is a special case of \cite[Theorem 4.1]{Sun}. It also can be deduced directly from the theorem of Griffiths-Harris.)  \end{example}

\begin{remark} \textbf{(More general degeneracy loci).} An analogous statement -- with essentially the same proof -- holds for more general determinantal loci. Specifically, consider vector bundles $V$ and $E$ of ranks $e+1$ and $n+e$, and suppose that $w : V^* \lra E$ is a homomorphism that drops rank simply on a reduced finite set $Z \subseteq X$. As above, this gives rise to a surjection
\[  E^* \lra  V \lra B \otimes \OO_Z \lra 0. \]
Then the statement of Theorem \ref{CB.for.Det.Loci} remains valid provided that one takes $L = \det E + \det V$ and
\[  c_1 \ = \ \dim \coker \Big(   \HH{0}{X}{S^{n-1}V}\lra \HH{0}{X}{B^{n-1} \otimes \OO_Z}\Big). \ \hskip 10pt\qed\]
\end{remark}

\begin{remark} \textbf{(Non-reduced degeneracy loci).}
\label{Caveat}
One can remove the hypothesis that the finite set  $Z$ be reduced by assuming instead that the map $w$ (and hence also $v$) drops rank by exactly one at every point $x \in Z$. In this case $\coker  (u)$ is still a line bundle on the degeneracy scheme $Z$ defined by the vanishing of the maximal minors of $u$, as one sees by locally pulling back $u$ from the space of all matrices. Then $Z$ again embeds in $Y$, where it is the zero-locus of a section of a vector bundle. In particular $Z$ is a local complete intersection scheme, and therefore Gorenstein, and one can proceed as in Remark \ref{Non.Reduced.Zeroes}. \qed \end{remark}

\section{Excess Vanishing}

 This section is devoted to the proofs of Theorems \ref{ExcessCB.Intro.B} and \ref{CB.Excess.TV.Intro} from the Introduction. 
 
 We begin with a quick review of the basic facts about multiplier ideals, referring to \cite[Chapter 9]{PAG} or \cite{PCMI} for details. Let $X$ be a smooth complex variety of dimension $n$, and let $\frb \subseteq \OO_X$  be a coherent sheaf of ideals on $X$. One associates to $\frb$ and its powers a \textit{multiplier ideal sheaf}
 \[ \MI{\frb^m} \, = \, \MMI{X}{\frb^m} \, \subseteq \, \OO_X,  \]
 as follows.
 Start by forming a log resolution
$  \mu : X^\pr \lra X$ of $\frb$, i.e. a proper birational map, with $X^\pr$ smooth, such that 
\[   \frb \cdot \OO_{X^\pr}  \, = \, \OO_{X^\pr}(-B) \]
where $B$ is an effective divisor on $X^\pr$ such that $B + K_{X^\pr / X}$ has simple normal crossing support. One then takes
\begin{equation} \label{Def.of.MI}
\MI{\frb^m} \ = \ \mu_* \left ( \OO_X \left ( K_{X^\pr / X} - mB\right)\right) . \footnote{More generally one can define $\MI{\frb^c}$ for any rational number $c > 0$, but we will not requiren this.}
\end{equation}
 One shows that the definition is independent of the choice of log-resolution.  The intuition is that these multiplier ideals measure the singularities of functions $f \in \frb$, with ``deeper" ideals corresponding to 
``greater singularities."

Multiplier ideals satisfy many  pleasant properties. We mention two here by way of orienting the reader. First, keeping notation as in \eqref{Def.of.MI}, one has:
\begin{equation} \label{Local.Vanishing}   R^j\mu_* \left ( \OO_X \left ( K_{X^\pr / X} - mB\right)\right) \ = \ 0 \ \ \text{for  } j > 0.\end{equation}
 This is known as the local vanishing theorem for multiplier ideals, and it guarantees in effect that the $\MI{\frb^m}$ will be particularly well-behaved. Secondly, Skoda's theorem states that
 \[   \MI{\frb^n} \ = \ \frb \cdot \MI{\frb^{n-1}}, \]
 where as always $n = \dim X$. In particular, $\MI{\frb^n} \subseteq \frb$, meaning that the multiplier ideals appearing in Theorems \ref{Excess.CB.1} and \ref{Excess.CB.2} below are at least as deep as $\frb$ itself. 
 
 \begin{example} \textbf{(Multiplier ideals of smooth subvarieties).}\label{MI.Smooth.Vars}
 Suppose that $W \subseteq X$ is a smooth subvariety of dimension $w$. Let us show that
 \[  \MMI{X}{I_W^m} \ = \  I_W^d, \]
 where $d = \max\{ 0, w + 1 + (m - n)\}$. 
 In fact, the blowing up $\mu  : X^\pr = \Bl_W(X)\lra X$  of $W$ is a log resolution, with
 \[ K_{X^\pr / X} = (n-w-1)E  \ \ , \ \ I_W \cdot \OO_{X^\pr} = \OO_{X^\pr }(-E),
 \]
 $E \subseteq \OO_{X^\pr}$ being the exceptional divisor. Therefore
 \[  \MI{I_W^m} \ = \ \mu_* \left ( \OO_X \left ( \left ( n - w - 1 -m \right) E \, \right)\right)  \ = \ I_W^d,\]
as claimed.
 \end{example}

 We now come to our main results. Let $X$ be smooth complex projective variety of dimension $n$, and $E$ a rank $n$ vector bundle on $X$ with $\det E = L$. 
 
 \begin{theorem}  \label{Excess.CB.1}
 Let $ s \in \GGG{X}{E}$ be a section whose zero-scheme is defined by the ideal
 \[   \frb \cdot \II_Z \ \subseteq \OO_X, \]
 where $Z \subseteq X$ is a non-empty reduced finite set, and $\frb \subseteq \OO_X$ is an arbitrary ideal whose zero-locus is disjoint from $Z$. Suppose that 
\[  h \, \in \, \Gamma \Big( X \, , \, \OO_X(K_X + L) \otimes \MI{\frb^n} \Big) \]
is a section vanishing at all but one of the points of $Z$. Then $h$ vanishes on the remaining point as well.
\end{theorem}

\begin{theorem} \label{Excess.CB.2}
In the setting of Theorem \ref{Excess.CB.1}, write $Z = Z_1 \sqcup Z_2$ as the disjoint union of two non-empty subsets, and fix an arbitrary line bundle $A$ on $X$. Write
\[ \begin{gathered}
v_1 \ = \ \dim  \coker \Big( H^0(A) \lra H^0\big(A \otimes \OO_{Z_1}\big) \Big) \\
v_2 \ = \ \dim \coker \Big( H^0\big( I_Z (K_X + L - A)\otimes \MI{
\frb^n}\big) \hookrightarrow H^0\big(I_{Z_2}(K_X + L - A) \otimes \MI{\frb^n}\big) \Big).
\end{gathered}
\]
Then $v_2 \le v_1$.
\end{theorem}

\noi Observe that Theorem \ref{ExcessCB.Intro.A}
 from the Introduction follows from \ref{Excess.CB.1}
 together with Example \ref{MI.Smooth.Vars}. 
 
 \begin{remark} \textbf{(Non-reduced zeros).} \label{Non.Red.Excess}
Provided that one proceeds as in Remark \ref{Non.Reduced.Zeroes}, one can remove the assumption that the zero-scheme $Z$ be reduced. \qed
 \end{remark}
 
 \begin{proof}[Proof of Theorems \ref{Excess.CB.1} and \ref{Excess.CB.2}] We will deduce both results from the corresponding statements in \S 1. Specifically, let 
$  \mu : X^\pr \lra X$ be a log resolution of $\frb$, with \[ \frb \cdot \OO_{X^\pr} \, = \, \OO_{X^\pr}(-B), \] where $B$ is an effective divisor on $X^\pr$ with SNC support. We may and do suppose that $\mu$ is constructed by a sequence of blowings up over Zeroes$(\frb)$, so that in particular $\mu$ is an isomorphism over a neighborhood of $Z$. Therefore $Z$ embeds naturally as a subset $Z^\pr \subseteq X^\pr$. 

By assumption the image of the natural mapping $E^* \lra \OO_X$ determined by $s$ is the ideal $\frb \cdot \II_{Z/X}$. This mapping pulls back to a surjection
\[   
\xymatrix{
\mu^* E^* \ar@{->>}[r] & \OO_{X^\pr} (-B) \cdot \II_{Z^\pr/X^\pr}. }
\]
In particular, setting \[ E^\pr \ =\  \mu^* E \otimes \OO_{X^\pr}(-B), \] $s$ gives rise to a section
$s^\pr\in \GGG{X^\pr}{ E^\pr}$ vanishing exactly on $Z^\pr \subseteq X^\pr$. 

For \ref{Excess.CB.1}, we apply  Theorem \ref{GH.CB.Thm} to this section. That result asserts that every section
\[    h^\pr \ \in \ \GGG{X^\pr}{\OO_{X^\pr}(K_{X^\pr} + \det E^\pr)} \] vanishing at all but one of the points of $Z^\pr$ also vanshes at the remaining one. But observe that
\[   K_{X^\pr} + \det E^\pr \ \lin (K_{X^\pr /X} - nB) + \mu^*(K_X + \det E).\]
Therefore
\[   \mu_* \big( \OO_{X^\pr}(K_{X^\pr} + \det E^\pr )\big) \ = \ \OO_X(K_X + \det E) \otimes \MI{\frb^n},
 \]
and in particular
\[
\GGG{X}{\OO_{X}(K_{X} + \det E) \otimes \MI{\frb^n}} \ = \ \GGG{X^\pr}{\OO_{X^\pr}(K_{X^\pr} + \det E^\pr)}.
\]
Recalling that $\mu$ is an isomorphism over a neighborhood of $Z$, Theorem \ref{Excess.CB.1} follows.

Theorem \ref{Excess.CB.2} follows in a similar manner from Theorem \ref{TV.Type.Statement}. Staying in the same setting, write $Z^\pr = Z_1^\pr \sqcup Z_2^\pr$ for the decomposition of $Z^\pr$ determined by $Z_1$ and $Z_2$, and put $A^\pr = \mu^* A$. Then 
\[
\HH{0}{X^\pr}{A^\pr} \ = \ \HH{0}{X}{A}, 
\]
and hence 
\[  v_1 \ = \ \dim  \coker \Big( H^0(A^\pr) \lra H^0\big(A^\pr \otimes \OO_{Z_1^\pr}\big) \Big).
\]
Moreover
\[
\begin{gathered}
\HH{0}{X}{ I_Z (K_X + L - A)\otimes \MI{
\frb^n} } \ = \ \HH{0}{X^\pr}{ I_{Z^\pr} (K_{X^\pr} + \det E^\pr - A^\pr)}
  \\ \HH{0}{X}{ I_{Z_2}(K_X + L - A)\otimes \MI{
\frb^n} } \ = \ \HH{0}{X^\pr}{ I_{Z_2^\pr} (K_{X^\pr} + \det E^\pr - A^\pr)},
\end{gathered}
\]
and hence \ref{TV.Type.Statement} yields \ref{Excess.CB.2}. 
\end{proof}

\begin{example} \textbf{(Skoda--Koszul complex).} \label{Skoda.Complex}
Let $ s \in \GGG{X}{E}$ be as in the hypothesis of Theorems \ref{Excess.CB.1} and \ref{Excess.CB.2}. If $\dim \textnormal{Zeroes}(s) \ge 1$, then the Koszul complex determined by $s$ is not exact. However we assert that the Koszul complex determined by $s$ contains an \textit{exact} subcomplex
 \[ 0 \lra \Lambda^n E^* \lra \Lambda^{n-1} E^* \otimes \MI{\frb} \lra \ldots \lra  E^* \otimes \MI{\frb^{n-1}} \lra \MI{\frb^n} \cdot \II_Z \lra 0.\tag{*} \]
In fact, consider the Koszul complex on $X^\pr$ arising from $s^\pr \in \GGG{X^\pr}{E^\pr}$:
\[0 \lra \Lambda^n  E^\pr {}^* \lra \Lambda^{n-1} E^\pr{}^* \lra \ldots \lra E^\pr{}^* \lra \II_{Z^\pr} \lra 0.\]
It is exact since $s^\pr$ vanishes in codimension $n$. 
Twisting  by $\OO_{X^\pr}(K_{X^\pr/X} - nB)$, one arrives at an exact sequence on $X^\pr$ with terms of the form
\[  \mu^* \big(  \Lambda^{n-i} E^* \big) \otimes \OO_{X^\pr}(K_{X^\pr /X} -iB). \]
These have vanishing higher direct images thanks to  \eqref{Local.Vanishing}, and it follows that the direct image of the twisted Kosul complex remains exact. But
\[
\mu_* \left( \mu^*\left(\Lambda^{n-i} E^* \right) \otimes \OO_{X^\pr}(K_{X^\pr /X} -iB)\right) \ = \ \Lambda^{n-i} E^* \otimes \MI{\frb^{i}},
\]
yielding the  long exact sequence (*). \qed 
\end{example}

\begin{example} \textbf{(Statements of Li-type).} \label{Li-Type.Statement} By adjusting the numerics, one can deduce from Theorem \ref{ExcessCB.Intro.A} statements closer to 
 the spirit of \cite{Li}. For example, suppose that $W \subseteq \PP^n$ is a smooth variety of dimension $w$ that is cut out scheme-theoretically by hypersurfaces of degree $e$, and consider hypersurfaces 
$D_1, \ldots, D_n$ of degrees $d_1, \ldots , d_n$ such that
\[   D_1 \cap \ldots \cap D_n \ =_{\text{scheme-theoretically}} W \sqcup Z, \]
where $Z$ is a non-empty reduced finite set. Then any hypersurface $H$ with
\[   \deg H \ = \ \Big( \sum d_i \Big) - (n+1) - w\cdot e \]
passing through $W$ and all but one of the points of $Z$ must also pass through the remaining point of $Z$. For instance, in the example from the Introduction this applies if $H$ is a surface of degree $d-2$ passing through $C$ and all but one of the $(d-2)$ points of $Z$. (In fact, we may choose hypersurfaces $H_1, \ldots, H_w$ of degree $e$ passing through $W$ but missing every point of $Z$. The assertion then follows by applying Theorem \ref{ExcessCB.Intro.A} to the hypersurface $H + H_1 + \ldots + H_w$.) \qed
\end{example}

\begin{example} \textbf{(Fibres of rational coverings of projective space).} Let $X$ be a smooth projective variety of dimension $n$,  let 
\[   f : X \dra \PP^n\]
be a generically finite rational mapping, and let $Z \subseteq X$ be a generic fibre of $f$. Denote by $D \subseteq X$ the proper transform of a hyperplane in $\PP^n$, so that $f$ is defined by a linear series $\linser{V} \subseteq \linser{D}$, with base ideal $\frb = \frb\big( \linser{V} \big)$. Then $Z$ is the isolated zero-locus of a section of $\OO_X^n(D)$ vanishing also along $\frb$, so Theorem \ref{Excess.CB.1} implies:

\parbox{.6in}{ (*)}  \parbox{4.25in}{Any section of $\OO_X(K_X+nD) \otimes \MI{\frb^n}$ vanishing at all but one of the points of $Z$ also vanishes at the remaining one. 
} 

\noi Observe that we can find a section $t \in \GGG{X}{\OO_X(D) \otimes \frb}$ not vanishing at any point of $Z$, and then multiplication by $t^n$ determines an embedding
\[  \OO_X(K_X) \hookrightarrow \OO_X\big(K_X + nD\big) \otimes \MI{\frb^n} \]
that is an isomorphism along $Z$. So we recover the statement --  proved using Mumford's trace in \cite{BCD} -- that $Z$ satisfies the Cayley-Bacharach property with respect to $\linser{K_X}$. 
However (*) is a priori stronger since  $\linser{\OO_X(K_X+nD) \otimes \MI{\frb^n}}$ is typically larger than $\linser{K_X}$. The statements for $K_X$ are used  in  \cite{BCD} and \cite{BDELU} to study the \textit{degree of irrationality} of $X$ -- ie the least degree of a covering $f : X\dra \PP^n$. It would be interesting to know if (*) can lead to any improvements.   In a similar vein, given an arbitrary line bundle $A$, Theorem \ref{Excess.CB.2} leads to a statement involving the linear series $\linser{K_X - A}$ whose formulation we leave to the reader.  \qed
\end{example}

\begin{remark} \textbf{(Degeneracy loci).} We do not know whether or how one can generalize Theorem \ref{CB.for.Det.Loci}
 to the case of vector bundle maps with excess degeneracies.   \qed \end{remark}

\appendix
\section  {Proof of Lemma \ref{A.Dual.Lemma}} \label{Appendix.A}

We sketch here one way to verify the identification asserted in Lemma \ref{A.Dual.Lemma}. For clarity we work in a slightly more general setting.

Consider then a finite subscheme $Z \subseteq X$ of a smooth  projective variety of dimension $n$. Fix a locally free acyclic complex 
\[
 \ 0 \lra L_{-n} \lra L_{-n + 1} \lra \ldots \lra L_{-1} \lra L_0  \lra 0, \tag{$L_\bullet$}
\]
with $L_0 = \OO_X$ resolving $\OO_Z$ as an $\OO_X$-module, so that $\OO_Z = \mathcal{H}^0(L_\bullet)$. Let $A$ be an arbitrary line bundle on  $X$, and consider the complex
$  {M}_{\bullet} = \mathbf{D}_X({L_\bullet}\otimes A) =  {L}_\bullet^{\spcheck} \otimes A^* \otimes \omega_X[n]$. This has the form
\[
0 \lra M_{-n} \lra M_{-n+1} \lra \ldots \lra M_{-1} \lra M_{0} \lra 0, \tag{${M}_\bullet$}
\]
where $M_{-n + i} = L_i ^*\otimes \omega_X \otimes A^*$, and ${M}_\bullet$ is an acyclic resolution of  the sheaf $\mathit{Ext}^n(A \otimes \OO_Z, \omega_X)$. 
Breaking the long exact sequence
\[
\footnotesize
0 
 \to L_0^* \otimes \omega_X \otimes A^* \to L_1^* \otimes \omega_X \otimes A^*\to \ldots \to L_{-n}^* \otimes \omega_X \otimes A^* \to\mathit{Ext}^n(A \otimes \OO_Z, \omega_X) \to 0
\]
\normalsize
into short exact sequences and taking cohomology, and recalling that $L_0 = \OO_X$, one arrives at a homomorphism:
\[
\delta \, : \, \HH{0}{X}{\mathit{Ext}^n(A \otimes \OO_Z, \omega_X) } \lra \HH{n}{X}{\omega_X \otimes A^*}
\]
Now
\[ \HH{0}{X}{\mathit{Ext}^n(A \otimes \OO_Z, \omega_X) } \ = \ \textnormal{Ext}^n\big( A \otimes \OO_Z, \omega_X \big) \]
is Grothendieck dual to $\HH{0}{X}{A \otimes \OO_Z}$, and $\HH{n}{X}{\omega_X \otimes A^*}$ is dual to $\HH{0}{X}{A}$, 
so $\delta$ is identified with a mapping
\[
\delta^\pr : \HH{0}{X}{A \otimes \OO_Z}^* \lra \HH{0}{X}{A}^*.
\]
\begin{proposition}
The homomorphism $\delta^\pr$ is dual to the  restriction mapping 
\[  \rho : \HH{0}{X}{A} \lra \HH{0}{X}{A \otimes \OO_Z}.\]
\end{proposition}
\noi The first statement in Lemma \ref{A.Dual.Lemma} follows from this together with the self-duality of the Koszul complex. We leave the second statement -- involving the subscheme $Z_1$ of $Z$ -- to the reader.
\begin{proof} 
This follows from the functoriality of  Grothendieck--Verdier duality. Specifically, we view  $L_\bullet$ as representing $\OO_Z$ in D$^b(X)$. Recalling again that $L_0 = \OO_X$, the restriction $A \lra A \otimes \OO_Z$ is represented by the natural map of complexes
$  u :  A[0] \lra L_{\bullet} \otimes A$, and $\rho: H^0(A) \lra H^0(A \otimes \OO_Z)$ is given by the resulting homomorphism
\[
\mathbf{H}^0(u):\bHH{0}{X}{A[0]} \lra \bHH{0}{X}{L_\bullet\otimes A}
\]
on hypercohomology. Now the map $\mathbf{D}(u) : \mathbf{D}_X(L_\bullet \otimes A) \lra \mathbf{D}_X(A[0])$
determined by $u$ is the evident morphism 
\[ M_\bullet \lra M_{-n}[n] = \omega_X \otimes A^*[n].\] By Grothendieck--Verdier duality,  $\mathbf{H}^0(u)$ is dual to the resulting homomorphism
\[ \bHH{0}{X}{M_\bullet} \lra \bHH{0}{X}{\omega_X \otimes A^* [n]} .\]
The Proposition is therefore  a consequence of the following 
\end{proof}

\begin{lemma}
Consider an acyclic complex
\[
0 \lra M_{-n} \lra M_{-n+1} \lra \ldots \lra M_{-1} \lra M_{0} \lra 0, \tag{${M}_\bullet$}
\]
resolving a sheaf $\mathcal{F}$ on the $n$-dimensional projective variety $X$. Then the mapping
\[ 
\HH{0}{X}{\mathcal{F}} \lra \HH{n}{X}{M_{-n}} 
\]
determined by splitting $({M}_\bullet)$ into short exact sequences and taking cohomology coincides with the induced map on hypercohomology
\[
\bHH{0}{X}{{M}_\bullet} \lra \bHH{0}{X}{M_{-n}[n]}. 
\]
\end{lemma}
\begin{proof}
Denote by $(M_{\bullet})_{\le -i}$ the naive truncation of $M_\bullet$, obtained by replacing $M_j$ with the zero sheaf for $j > -i$. There is an evident map $(M_{\bullet})_{\le -i} \lra (M_{\bullet})_{\le -(i+1)}$ of complexes that induces on $\mathbf{H}^0$ the connecting homomorphism
\[
H^i\Big( \, X \, , \, \mathcal{H}^{-i}\big((M_{\bullet})_{\le -i}\big)\Big) \lra H^{i+1}\Big(  \, X \, , \ \mathcal{H}^{-i-1}\big((M_{\bullet})_{\le -(i+1)}\big)\Big)
\]
determined by the short exact sequence of sheaves
\[  0 \lra  \mathcal{H}^{-i-1}\big((M_{\bullet})_{\le -(i+1)}\big) \lra M_{-i} \lra  \mathcal{H}^{-i}\big((M_{\bullet})_{\le -i}\big) \lra 0. \]
Applying this remark successively starting with $i = 0$, the Lemma follows.
\end{proof}

 %
 %
 %
 %

 \end{document}